\documentclass[12pt]{amsart}
\usepackage[letterpaper,margin=1in]{geometry}
\usepackage{latexsym,amsfonts,amsmath,amssymb,mathtools,amsthm}
\usepackage{hyperref,enumitem}
\usepackage{mathrsfs,amsbsy}
\hypersetup{colorlinks=true,linkcolor=blue,linktoc=all}
\usepackage[backend=biber,style=alphabetic,sorting=nyt,maxbibnames=99]{biblatex}
\renewbibmacro{in:}{}
\newtheorem{thm}{Theorem}[section]
\newtheorem{lemma}[thm]{Lemma}
\newtheorem{defn}[thm]{Definition}
\newtheorem{prop}[thm]{Proposition}
\newtheorem{assumption}[thm]{Assumption}
\newcommand{\R}{\mathbb R}
\newcommand{\C}{\mathcal C}
\newcommand{\I}{\mathcal I}
\newcommand{\V}{\mathcal V}
\newcommand{\OO}{\mathcal O}
\newcommand{\A}{\mathscr A}
\newcommand{\s}{\subset}
\newcommand{\sm}{\setminus}
\newcommand{\vphi}{\varphi}
\newcommand{\into}{\lrcorner \,}
\DeclareMathOperator{\WF}{WF}
\DeclareMathOperator{\inj}{inj}
\DeclareMathOperator{\pv}{p{.}v.}
\DeclareMathOperator{\supp}{supp}
\DeclareMathOperator{\Range}{Range}
\title[Microlocal Analysis of the L\'evy Generator]{A Microlocal Analysis of the L\'evy Generator with Conjugate Points}
\author[K.\@ Tully]{Kevin Tully}
\date{}
\begin{document}
\maketitle


\begin{abstract}
    We analyze the microlocal structure of the infinitesimal generator of a L\'evy process on a closed Riemannian manifold when conjugate points are allowed. We show that if there are no singular conjugate pairs, then the infinitesimal generator can be written as a sum of pseudodifferential operators and Fourier integral operators. This extends and unifies known results for the flat torus, the sphere, and Anosov manifolds.
\end{abstract}


\section{Introduction}

In this paper, we will study the infinitesimal generator of a L\'evy process on a Riemannian manifold $(M, g)$. Loosely speaking, a L\'evy process on a manifold looks like Brownian motion interlaced with jumps along geodesics at random times. The infinitesimal generator encodes certain information about the process and reflects the geometry of $(M, g)$. For example, the generator of the most famous L\'evy process, Brownian motion, is the Laplace-Beltrami operator. Though Brownian motion on manifolds is well studied \cite{elworthy1988geometric, hsu2002stochastic}, the theory for more general L\'evy processes has received less attention. Hunt introduced L\'evy processes on Lie groups \cite{hunt1956semi}, while Gangolli initiated the study of symmetric spaces \cite{gangolli1964isotropic, gangolli1965sample}. Applebaum and Estrade were the first to construct L\'evy processes on arbitrary Riemannian manifolds, under a natural assumption on the L\'evy measure \cite{applebaum2000isotropic}. Our goal is to analyze the microlocal structure of the infinitesimal generator of this L\'evy process.

One motivation for studying L\'evy processes on manifolds is to shed light on the L\'evy flight foraging hypothesis, which is the foundation of several biological models \cite{viswanathan1996levy, bressloff2013stochastic} and search algorithms \cite{yang2009cuckoo, yang2010eagle, heidari2017efficient, kaidi2022dynamic}. This controversial hypothesis claims that L\'evy processes are a better model of animal foraging behavior than Brownian motion, in the sense that they can optimize search efficiencies \cite{shlesinger1986levy, viswanathan1999optimizing}. Even though the underlying geometry is often curved, until recently this topic had only been studied in Euclidean space. The expected time for a pure jump L\'evy process to find a small target on a manifold was first studied in \cite{chaubet2022geodesic}, while \cite{nursultanov2021mean, nursultanov2022narrow, nursultanov2023narrow} considered Brownian motion. A numerical comparison of \cite{chaubet2022geodesic} and \cite{nursultanov2023narrow} was performed in \cite{tzou2023challenging}, confirming that Brownian motion might be the faster search strategy for small targets on the $2$-torus. This suggests that the underlying geometry could be an important factor in determining whether the L\'evy flight foraging hypothesis is valid in a given context. 

In \cite{chaubet2022geodesic}, conjugate points are one of the main geometric influences on the expected stopping time. Specifically, the authors show that on the sphere the expected stopping time exhibits singular behavior at antipodal points, but no such anomaly occurs on the flat torus or Anosov manifolds. (We will call $(M, g)$ an Anosov manifold if its geodesic flow is an Anosov flow on its unit sphere bundle \cite{anosov1969geodesic, knieper2002hyperbolic}, or, equivalently, if $g$ lies in the $C^2$ interior of the set of metrics on $M$ without conjugate points \cite{ruggiero1991creation}.) Conjugate points exert influence on the expected stopping time through the microlocal structure of the L\'evy generator (i.e., the infinitesimal generator of the pure jump L\'evy process). On the sphere the L\'evy generator decomposes into a pseudodifferential operator and a Fourier integral operator, while on the flat torus and Anosov manifolds it is simply a pseudodifferential operator. Inspired by these results, as well as the microlocal analysis of the geodesic X-ray transform done in \cite{stefanov2012geodesic, monard2015geodesic, holman2018microlocal}, we prove a similar theorem in a more general geometric context.

Our main result, Theorem \ref{mainThm}, shows that if there are no singular conjugate pairs, then the L\'evy generator equals a sum of pseudodifferential operators and Fourier integral operators. Each Fourier integral operator is associated with conjugate pairs of a given order, and the order of the operator depends on the order of those conjugate pairs and the dimension of $M$. The canonical relation of each Fourier integral operator is related to the geometry of the set of conjugate pairs. When $(M, g)$ is the sphere or an Anosov manifold we recover the results of \cite{chaubet2022geodesic}. Our theorem also covers all possibilities in two dimensions.

The structure of the paper is as follows. In Section \ref{preliminaries} we discuss the necessary background, including an initial decomposition of the L\'evy generator into pseudodifferential operators and a remainder term. In Section \ref{mainSection} we prove our main result by showing that the remainder term equals a smoothing operator plus a sum of Fourier integral operators.


\subsection*{Acknowledgements}

The author would like to thank Gunther Uhlmann for proposing this problem and for his consistent support, guidance, and patience. This material is based upon work supported by the National Science Foundation Graduate Research Fellowship under Grant No.\@ DGE-2140004.


\section{Preliminaries} \label{preliminaries}

This section introduces our notation, provides the main definitions, and recalls a result from \cite{chaubet2022geodesic} that will guide our analysis of the L\'evy generator. Throughout the paper, we will assume $(M, g)$ satisfies the following condition.

\begin{assumption}
    $(M, g)$ is a smooth, closed (that is, compact without boundary), and connected Riemannian manifold of dimension $n \geq 2$.
\end{assumption}


\subsection{The L\'evy Generator}

Before defining the L\'evy generator, let us specify our notation for vector bundles and introduce the exponential map. 

Unless stated otherwise, we will denote the projection map associated with a vector bundle by $\pi$ with a subscript indicating the total space of the bundle (e.g., $\pi_{TM}$ is the projection map from the tangent bundle $TM$ to $M$). When referring to a point in a vector bundle, one must decide whether to specify the base point. For example, should a point in $TM$ be denoted by $(x, v) \in TM$ or simply $v \in TM$? We will often omit the base point, but we will include it when there is potential for confusion. We do follow the convention that $(x, v) \in TM$ means $v \in TM$ and $x = \pi_{TM}(v)$, and similarly for points in other vector bundles.

For each $v \in TM$, let $\gamma_v$ be the maximal geodesic with initial data $\dot\gamma_v(0) = v$. Since we view $\dot\gamma_v(0)$ as a point in $TM$, this initial data includes (loosely speaking) both the initial position and velocity of the geodesic. Each geodesic $\gamma_v$ is defined on all of $\R$ because $M$ is compact. Thus we can define the exponential map $\exp : TM \to M$ by
\begin{equation*}
    \exp(v) = \gamma_v(1).
\end{equation*}
For each $x \in M$, the restriction of the exponential map to $T_x M$ will be denoted by $\exp_x$.

Fix $\alpha \in (0, 1)$. It is well known that the fractional Laplacian of order $\alpha$ is the infinitesimal generator of a $2\alpha$-stable jump process on Euclidean space. To ensure that the L\'evy generator is consistent with the fractional Laplacian, we will include the constant
\begin{equation*}
    C_{n, \alpha} = \frac{4^\alpha\Gamma(n / 2 + \alpha)}{\pi^{n / 2}|\Gamma(-\alpha)|}.
\end{equation*}

\begin{defn} \label{Adef}
    The \textbf{L\'evy generator $\pmb\A$} is defined for $u \in C^\infty(M)$ and $x \in M$ by
    \begin{equation*} 
        (\A u)(x) = C_{n, \alpha} \, \pv\int_{T_x M \sm \{ 0 \}} \frac{u(\exp_x(v)) - u(x)}{|v|_g^{n + 2\alpha}} \, dT_x(v),
    \end{equation*}
    where $dT_x$ is the Riemannian density of $T_x M$ induced by the metric $g|_{T_x M}$.
\end{defn}


\subsection{Averaging Along the Geodesic Flow}

This subsection introduces one of the operators which will appear in our initial decomposition of the L\'evy generator.

Let $SM$ be the unit sphere bundle over $M$. Its fiber over a point $x \in M$ is
\begin{equation*}
    S_x M := \{ v \in T_x M : |v|_g = 1 \}.
\end{equation*}
Since $M$ is compact and $n$-dimensional, $SM$ is a compact manifold of dimension $2n - 1$, and it is an embedded submanifold of $TM$. We will write $\pi$ rather than $\pi_{SM}$ for the projection map $SM \to M$, and we will use $\iota_{SM}$ for the inclusion map $\iota_{SM} : SM \hookrightarrow TM$.

The geodesic flow on $SM$ is the smooth map $\Phi : SM \times \R \to SM$ defined by
\begin{equation*}
    \Phi(v, s) = \dot\gamma_v(s).
\end{equation*}
It is a smooth submersion because $\Phi(\cdot, s) : SM \to SM$ is a diffeomorphism for each $s \in \R$. We will denote the differential of the map $\Phi(\cdot, s)$ at a vector $v \in SM$ by $d_v\Phi|_{(v, s)}$. Also, the geodesic flow on $TM$ will be denoted by $\tilde\Phi$.

Since $(M, g)$ is a closed Riemannian manifold, its injectivity radius $r_{\inj}$ is positive and finite. Choose a bump function $\chi \in C_c^\infty(\R)$ satisfying $\chi(s) = 1$ for $|s| < r_{\inj}^2 / 4$ and $\chi(s) = 0$ for $|s| > r_{\inj}^2 / 2$. Let
\begin{equation*}
    a(s) = 
    \begin{cases}
        (1 - \chi(s^2))s^{-1 - 2\alpha} & \text{if} \ s \geq 0, \\
        0 & \text{if} \ s \leq 0.
    \end{cases}
\end{equation*}
Then we can define an operator $R_a : C^\infty(SM) \to C(SM)$ by
\begin{equation} \label{R_a}
    (R_a f)(v) = C_{n, \alpha} \int_\R a(s)f(\Phi(v, s)) \, ds.
\end{equation}
We may think of $R_a f$ as a certain average of $f$ along the geodesic flow. This operator will play a crucial role in our analysis of the L\'evy generator.


\subsection{Pushforward and Pullback by a Smooth Submersion}

Next we will define the pushforward and pullback and set our notation for Fourier integral operators.

\begin{defn} \label{pushPullDef}
    Let $X$ and $Y$ be smooth manifolds of dimension $n_X$ and $n_Y$, respectively, with smooth positive densities $dx$ and $dy$. Let $F : X \to Y$ be a smooth submersion. Then the \textbf{pushforward} $F_* : C_c^\infty(X) \to C_c^\infty(Y)$ and the \textbf{pullback} $F^* : C_c^\infty(Y) \to C^\infty(X)$ are defined by the requirement that
    \begin{equation*}
        \int_Y (F_*\vphi)(y)\psi(y) \, dy = \int_X \vphi(x)\psi(F(x)) \, dx = \int_X \vphi(x)(F^*\psi)(x) \, dx
    \end{equation*}
    for all $\vphi \in C_c^\infty(X)$ and $\psi \in C_c^\infty(Y)$. 
\end{defn}

Explicitly, the pushforward by $F$ integrates over the level sets of $F$, while the pullback precomposes with $F$. 

We would like to use Definition \ref{pushPullDef} to express $R_a$ as a composition of simpler operators. Let $\tilde a^m$ be the operator which multiplies by the smooth function
\begin{equation*}
    \tilde a(v, s) := C_{n, \alpha} \, a(s).
\end{equation*}
Then $\tilde a^m$ is a properly supported pseudodifferential operator of order $0$ on $SM \times \R$. Let
\begin{equation*}
    p : SM \times \R \to SM
\end{equation*}
be the projection map. Since $p_*$ integrates a function over $\R$, \eqref{R_a} suggests that $R_a$ equals the composition $p_* \circ \tilde a^m \circ \Phi^*$. This is not quite correct since $\tilde a^m \circ \Phi^*f$ may not have compact support, even if $f \in C^\infty(SM)$. But we can use a partition of unity to define the composition of $p_*$ and $\tilde a^m \circ \Phi^*$, and this is how we will view the operator $R_a$.

Recall that if $F : X \to Y$ is a smooth map between smooth manifolds, then for each $x \in X$ the differential $dF|_x : T_x X \to T_{F(x)} Y$ yields a dual linear map
\begin{equation*}
    dF|_x^t : T_{F(x)}^*Y \to T_x^*X.
\end{equation*}
Since $dF|_x^t$ is given in coordinates by the transpose of the matrix of $dF|_x$, if $F$ is a smooth submersion then $dF|_x^t$ is injective for all $x \in X$.

It is known, going back at least to \cite{guillemin1975fourier}, that the pushforward and pullback by a smooth submersion are both Fourier integral operators. We will state the version from Lemma $1$ of \cite{holman2018microlocal} to have this result in the precise format we need.

\begin{lemma} \label{pushPullLemma}
    Suppose we are in the setting of Definition \ref{pushPullDef}. Then the pushforward $F_*$ and the pullback $F^*$ are both Fourier integral operators of order $(n_Y - n_X) / 4$. The canonical relation of $F_*$ is
    \begin{equation*}
        C_{F_*} = \left\{ \left( \eta, dF|_x^t \, \eta \right) : x \in X, \ \eta \in T_{F(x)}^* Y \sm \{ 0 \} \right\},
    \end{equation*}
    while the canonical relation of $F^*$ is
    \begin{equation*}
        C_{F^*} = \left\{ \left( dF|_x^t \, \eta, \eta \right) : x \in X, \ \eta \in T_{F(x)}^* Y \sm \{ 0 \} \right\}.
    \end{equation*}
\end{lemma}

We will use the notation of \cite{hormander2009analysis} for Fourier integral operators. Given smooth manifolds $X$ and $Y$, the set $\I^m(X \times Y, C^\prime)$ consists of all operators which map $\mathcal E^\prime(Y)$ to $\mathcal D^\prime(X)$, and whose Schwartz kernel is a Fourier integral of order $m$ with Lagrangian given by the twisted canonical relation $C^\prime$. Unlike \cite{hormander2009analysis}, we will allow $C$ to be merely a local canonical relation, meaning it can be immersed rather than embedded. 

In this setup, Lemma \ref{pushPullLemma} implies that 
\begin{align*}
    \Phi^* &\in \I^{-1 / 4} \big( (SM \times \R) \times SM, C_{\Phi^*}^\prime \big), \\
    p_* &\in \I^{-1 / 4} \big( SM \times (SM \times \R), C_{p_*}^\prime \big), \\
    \pi^* &\in \I^{(1 - n) / 4}(SM \times M, C_{\pi^*}^\prime), \\
    \pi_* &\in \I^{(1 - n) / 4}(M \times SM, C_{\pi_*}^\prime).
\end{align*}
In the next subsection, we will see that the key to determining the microlocal structure of $\A$ is to show that a certain composition involving these operators is a Fourier integral operator. Our main tool will be the clean intersection calculus for Fourier integral operators, whose statement may be found in \cite{duistermaat1975spectrum}, \cite{weinstein1975maslov}, and \cite{hormander2009analysis}.


\subsection{Initial Decomposition of the L\'evy Generator}

The next result, a consequence of Theorem $1.6$ in \cite{chaubet2022geodesic}, is the starting point for our microlocal analysis of $\A$. The idea of the proof is to split $\A$ into two parts: one where $\exp_x$ is injective, and a remainder term involving $\pi_*$, $R_a$, and $\pi^*$. This decomposition is unnecessary if $(M, g)$ is the flat torus, since $-\A$ is the fractional Laplace-Beltrami operator in that case (Theorem $1.4$ in \cite{chaubet2022geodesic}). Thus we may safely rule out that case throughout the paper.

\begin{thm} \label{initialDecomp}
    If $R_a$ preserves $C^\infty(SM)$, then
    \begin{equation*} 
        \A = \A_{2\alpha} + \A_0 + \pi_* \circ R_a \circ \pi^*,
    \end{equation*}
    where $\A_{2\alpha}$ (resp.\@ $\A_0$) is a pseudodifferential operator of order $2\alpha$ (resp.\@ $0$) on $M$.
\end{thm}

\begin{proof}
    Fix $u \in C^\infty(M)$ and $x \in M$. Write $\A = A_1 + A_2$, where
    \begin{align*}
        (A_1 u)(x) &= C_{n, \alpha} \pv\int_{T_x M \sm \{ 0 \}} \chi \left( |v|_g^2 \right) \frac{u(\exp_x(v)) - u(x)}{|v|_g^{n + 2\alpha}} \, dT_x(v), \\
        (A_2 u)(x) &= C_{n, \alpha} \int_{T_x M \sm \{ 0 \}} \left( 1 - \chi \left( |v|_g^2 \right) \right) \frac{u(\exp_x(v)) - u(x)}{|v|_g^{n + 2\alpha}} \, dT_x(v).  
    \end{align*}

    By our choice of $\chi$, the integrand of $A_1$ vanishes whenever $|v|_g^2 > r_{\inj}^2 / 2$, so we can make the change of variables $y = \exp_x(v)$. Then
    \begin{equation*}
        (A_1 u)(x) = C_{n, \alpha} \pv\int_M \chi(d_g(x, y)^2)\frac{u(y) - u(x)}{d_g(x, y)^{n + 2\alpha}}J(x, y) \, dV_g(y),
    \end{equation*}
    where $d_g(x, y)$ is the Riemannian distance from $x$ to $y$, $J(x, y)$ is the Jacobian determinant of the map $y \mapsto \exp_x^{-1}(y)$, and $dV_g$ is the Riemannian density of $(M, g)$. Hence $\A_{2\alpha} := A_1$ is a pseudodifferential operator of order $2\alpha$ on $M$.

    For $A_2$, use polar coordinates to write
    \begin{equation*}
        (A_2 u)(x) = C_{n, \alpha} \int_{S_x M} \int_0^\infty \left( 1 - \chi \left( s^2 \right) \right) \frac{u(\exp_x(sv)) - u(x)}{s^{1 + 2\alpha}} \, ds \, dS_x(v),
    \end{equation*}
    where $dS_x$ is the Riemannian density of $S_x M$ induced by the metric $g|_{S_x M}$. Since $R_a$ preserves $C^\infty(SM)$ and $\pi_*$ integrates over the fibers of $\pi$, we can write
    \begin{equation*}
        (A_2 u)(x) = (\pi_* \circ R_a \circ \pi^*u)(x) - (\pi_* \circ R_a \circ \pi^*1)(x)u(x).
    \end{equation*}
    Let $\A_0$ be the operator which multiplies by the constant function $-(\pi_* \circ R_a \circ \pi^*1)$. Since $\A_0$ is a pseudodifferential operator of order $0$ on $M$, this completes the proof.
\end{proof}

We will see in Theorem \ref{R_aThm} that $R_a$ is a Fourier integral operator, so it preserves $C^\infty(SM)$. Therefore the microlocal analysis of $\A$ boils down to showing that $\pi_* \circ R_a \circ \pi^*$ is a Fourier integral operator, a task which we will take up in the next section.


\section{Microlocal Structure of the L\'evy Generator} \label{mainSection}

This section contains our main result: a decomposition of $\A$ into pseudodifferential operators and Fourier integral operators. In light of Theorem \ref{initialDecomp}, a natural first step is to show that $R_a$ is a Fourier integral operator. Since we will eventually need to split $\pi_* \circ R_a \circ \pi^*$ into several pieces, we will actually introduce an arbitrary smooth function into $R_a$ and show this more general operator is a Fourier integral operator.


\subsection{$R_a$ is a Fourier Integral Operator}

Givan any $\psi \in C^\infty(SM \times \R)$, let
\begin{equation*}
    R_{a, \psi} = p_* \circ (\psi\tilde a)^m \circ \Phi^*,
\end{equation*}
where $(\psi\tilde a)^m$ is the operator which multiplies by $\psi\tilde a$. Since $(\psi\tilde a)^m$ is a properly supported pseudodifferential operator of order $0$ on $SM \times \R$ and $\Phi^*$ is in $\I^{-1 / 4} \big( (SM \times \R) \times SM, C_{\Phi^*}^\prime \big)$, the operator $(\psi\tilde a)^m \circ \Phi^*$ is also in $\I^{-1 / 4} \big( (SM \times \R) \times SM, C_{\Phi^*}^\prime \big)$. Therefore, if we can show that the clean intersection calculus applies to the composition of $p_*$ and $(\psi\tilde a)^m \circ \Phi^*$, then $R_{a, \psi}$ is a Fourier integral operator. This is the content of the following theorem. 

\begin{thm} \label{R_aThm}
    Let
    \begin{align*}
        C_{R_{a, \psi}} = \Big\{ (\xi, \tilde\xi) \in T^*SM \times T^*SM : & \ \exists \, s \in \R \ \text{such that}
        \\ & \ dp|_{(\pi_{T^*SM}(\xi), s)}^t \, \xi = d\Phi|_{(\pi_{T^*SM}(\xi), s)}^t \, \tilde\xi \Big\}.
    \end{align*}
    Then $R_{a, \psi} \in \I^{-1 / 2}(SM \times SM, C_{R_{a, \psi}}^\prime)$.
\end{thm}

\begin{proof}
    As noted above, we can use a partition of unity to define the composition of $p_*$ and $(\psi\tilde a)^m \circ \Phi^*$. Then by localizing and reducing to the case of two properly supported Fourier integral operators, the proof boils down to an analysis of the canonical relations.

    Since $C_{\Phi^*}$ is the canonical relation of $(\psi\tilde a)^m \circ \Phi^*$, Lemma \ref{pushPullLemma} implies that
    \begin{equation*}
        C_{R_{a, \psi}} = C_{p_*} \circ C_{\Phi^*}.
    \end{equation*}
    Therefore, by the clean intersection calculus, it is enough to prove the following:
    \begin{enumerate}
        \item[(i)] The intersection
        \begin{equation} \label{C1}
            C := (C_{p_*} \times C_{\Phi^*}) \cap \big( T^*SM \times \Delta \big( T^*(SM \times \R) \big) \times T^*SM \big)
        \end{equation}
        is clean in the sense that $C$ is an embedded submanifold, and at every point $c \in C$ the tangent space $T_c C$ equals the intersection of the tangent spaces of the two manifolds being intersected.
        \item[(ii)] The projection map $\pi_C : C \to T^*SM \times T^*SM$ is proper.
        \item[(iii)] For every $(\xi, \tilde\xi) \in T^*SM \times T^*SM$, the fiber $\pi_C^{-1}(\xi, \tilde\xi)$ is connected.
    \end{enumerate}
    If we are willing to work with local canonical relations, then point (iii) can be omitted. Point (iii) will hold in this case, but in later results it may not.

    Consider the map
    \begin{equation*}
        G : T^*SM \times \R \to T^*(SM \times \R)
    \end{equation*}
    defined by
    \begin{equation*} 
        G(\xi, s) = dp|_{(\pi_{T^*SM}(\xi), s)}^t \, \xi.
    \end{equation*}
    If we include the base points, then
    \begin{equation} \label{G}
        G(\pi_{T^*SM}(\xi), \xi, s) = \big( (\pi_{T^*SM}(\xi), s), (\xi, 0) \big),
    \end{equation}
    so $G$ is a smooth embedding. The injectivity of $G$ implies that $\pi_C$ is injective. Hence point (iii) holds, and assuming points (i) and (ii) are true the excess of the intersection \eqref{C1} is zero.

    To begin proving point (i), let $Z$ be the smooth rank-$(2n - 1)$ subbundle of $T^*(SM \times \R)$ whose fiber over each point $(v, s) \in SM \times \R$ is
    \begin{equation*}
        Z_{(v, s)} := \Range \left( d\Phi|_{(v, s)}^t \right).
    \end{equation*}
    In other words, $Z_{(v, s)}$ is the subspace of $T_{(v, s)}^*(SM \times \R)$ which is conormal to the kernel of $d\Phi|_{(v, s)}$. To see that $Z$ is a smooth subbundle, just use the rank theorem to write
    \begin{equation*}
        \Phi \left( y^1, \dots, y^{2n} \right) = \left( y^1, \dots, y^{2n - 1} \right)
    \end{equation*}
    locally, and note that $\left( dy^1, \dots, dy^{2n - 1} \right)$ is a smooth local frame for $Z$. Because $d\Phi|_{(v, s)}^t$ is a linear isomorphism from $T_{\Phi(v, s)}^*SM$ to $Z_{(v, s)}$, we can define a smooth map
    \begin{equation} \label{dPhi^-t}
        d\Phi^{-t} : Z \to T^*SM
    \end{equation}
    whose restriction to each fiber $Z_{(v, s)}$ is the inverse $(d\Phi|_{(v, s)}^t)^{-1} : Z_{(v, s)} \to T_{\Phi(v, s)}^*SM$.

    The domain of our smooth parametrization of $C$ will be the set
    \begin{equation*}
        \OO := \Range(G) \cap Z.
    \end{equation*}
    We claim that $\OO$ is an embedded submanifold of dimension $4n - 2$. Let
    \begin{equation} \label{q}
        q : Z \to \R
    \end{equation}
    be the restriction to $Z$ of the projection map
    \begin{equation*}
        T^*(SM \times \R) \ni \big( (v, s), (\xi, \sigma) \big) \mapsto \sigma.
    \end{equation*}
    Since $\OO$ equals the level set $q^{-1}(0)$, the claim is true if $dq|_\zeta$ is nonzero for all $\zeta \in Z$. Let $\pi_V(\zeta) = (v, s)$. Then $\zeta = d\Phi|_{(v, s)}^t \, \theta$ for some $\theta \in T_{\Phi(v, s)}^*SM$. Choose slice coordinates for $SM$ near $v$ and near $\Phi(v, s)$, and fix natural coordinates for $T^*SM$ associated with the latter. Then locally we can write $\Phi = \left( \Phi^1, \dots, \Phi^{2n - 1} \right)$ and $\theta = (\theta_1, \dots, \theta_{2n - 1})$. Since $\Phi(v, \cdot)$ is a unit-speed geodesic, we may suppose without loss of generality that 
    \begin{equation*}
        \left. \frac{d}{d\tilde s} \right|_{\tilde s = s} \Phi^1(v, \tilde s) \neq 0.
    \end{equation*}
    Define a curve $\beta = \left( \beta^1, \dots, \beta^{2n} \right)$ in $T^*SM$ as the composition of $d\Phi|_{(v, s)}^t$ and the curve
    \begin{equation*}
        \R \ni \tau \mapsto (\theta_1 + \tau, \theta_2, \dots, \theta_{2n - 1}) \in T_{\Phi(v, s)}^*SM.
    \end{equation*}
    Then $\beta$ is a smooth curve in $Z_{(v, s)}$ such that $\beta(0) = \zeta$. Moreover,
    \begin{equation} \label{dqNonzero}
        \left. \frac{d}{d\tau} \right|_{\tau = 0} (q \circ \beta)(\tau) = \left. \frac{d}{d\tau} \right|_{\tau = 0} \beta^{2n}(\tau) = \left. \frac{d}{d\tilde s} \right|_{\tilde s = s} \Phi^1(v, \tilde s) \neq 0.
    \end{equation}
    Therefore, as claimed, $\OO$ is an embedded submanifold of dimension $4n - 2$.

    Now we can use $\OO$ to parametrize $C$. Indeed, if $G_\xi^{-1}$ is the inverse of \eqref{G} composed with the projection onto the $\xi$ component, then the map
    \begin{equation*}
        P_C : \OO \to T^*SM \times \OO \times \OO \times T^*SM
    \end{equation*}
    defined by
    \begin{equation*}
        P_C(\zeta) = \left( G_\xi^{-1}(\zeta), \zeta, \zeta, d\Phi^{-t}\zeta \right)
    \end{equation*}
    is a smooth embedding, so $C$ is an embedded submanifold of dimension $4n - 2$.

    To finish proving point (i), fix $c \in C$. Since $T_c C$ is necessarily contained in the intersection of the tangent spaces of the manifolds on the right-hand side of \eqref{C1}, it is enough to show the reverse containment. Observe that $C_{p_*} \times C_{\Phi^*}$ can be parametrized by the map
    \begin{equation*}
        P_{C_{p_*} \times C_{\Phi^*}} : T^*SM \times \R \times Z \to T^*SM \times \Range(G) \times Z \times T^*SM
    \end{equation*}
    defined by
    \begin{equation*}
        P_{C_{p_*} \times C_{\Phi^*}}(\xi, s, \zeta) = \left( \xi, G(\xi, s), \zeta, d\Phi^{-t}\zeta \right).
    \end{equation*}
    Hence any vector $X \in T_c(C_{p_*} \times C_{\Phi^*})$ is the velocity of some smooth curve
    \begin{equation*}
        \R \ni \tau \mapsto (\xi(\tau), s(\tau), \zeta(\tau)) \in T^*SM \times \R \times Z,
    \end{equation*}
    meaning $P_{C_{p_*} \times C_{\Phi^*}}(\xi(0), s(0), \zeta(0)) = c$ and
    \begin{equation*}
        X = \left. \frac{d}{d\tau} \right|_{\tau = 0} P_{C_{p_*} \times C_{\Phi^*}}(\xi(\tau), s(\tau), \zeta(\tau)).
    \end{equation*}
    The vector $X$ is also in $T_c \big( T^*SM \times \Delta \big( T^*(SM \times \R) \big) \times T^*SM \big)$ if and only if
    \begin{equation*}
        \left. \frac{d}{d\tau} \right|_{\tau = 0} G(\xi(\tau), s(\tau)) = \left. \frac{d}{d\tau} \right|_{\tau = 0} \zeta(\tau).
    \end{equation*}
    Thus, in any local coordinates, $G(\xi, s)$ and $\zeta$ agree to first order at $\tau = 0$. Then \eqref{G} implies that the same is true of $\xi$ and $G_\xi^{-1}(\zeta)$. Hence
    \begin{align*}
        X &= \left. \frac{d}{d\tau} \right|_{\tau = 0} P_{C_{p_*} \times C_{\Phi^*}}(\xi(\tau), s(\tau), \zeta(\tau)) \\
        &= \left. \frac{d}{d\tau} \right|_{\tau = 0} P_C(\zeta(\tau)),
    \end{align*}
    which means $X \in T_c C$. This completes the proof of point (i).

    For point (ii), suppose we have a compact set $K \s T^*SM \times T^*SM$. Then there exists a constant $\rho > 0$ such that
    \begin{equation*}
        K \s S := \left\{ (\xi, \tilde\xi) \in T^*SM \times T^*SM : |\xi|_g + |\tilde\xi|_g \leq \rho \right\}.
    \end{equation*}
    Hence $P_C^{-1}(\pi_C^{-1}(S))$ is precisely the set
    \begin{equation*}
        \left\{ \zeta \in \OO : |G_\xi^{-1}(\zeta)|_g + |d\Phi^{-t}\zeta|_g \leq \rho \right\},
    \end{equation*}
    which is compact by continuity. Then $P_C^{-1}(\pi_C^{-1}(K))$ is compact, because it is a closed subset of the compact set $P_C^{-1}(\pi_C^{-1}(S))$. Since $P_C$ is a diffeomorphism onto $C$, this implies that $\pi_C$ is a proper map. Thus point (ii) holds.

    Since points (i)-(iii) hold and the excess is zero, the map $\pi_C \circ P_C$ is a smooth embedding, so its image $C_{R_{a, \psi}}$ is an embedded submanifold. Since $p_*$ and $(\psi\tilde a)^m \circ \Phi^*$ are Fourier integral operators of order $-1 / 4$, we conclude that $R_{a, \psi} \in \I^{-1 / 2}(SM \times SM, C_{R_{a, \psi}}^\prime)$.
\end{proof}

Because $R_{a, \psi}$ is a Fourier integral operator with canonical relation $C_{R_{a, \psi}}$, its twisted wave front set must be contained in $C_{R_{a, \psi}}$. In the next lemma, we will leverage the fact that $R_{a, \psi}$ is a composition of Fourier integral operators to say a bit more.

\begin{lemma} \label{R_aWF}
    The twisted wave front set of $R_{a, \psi}$ is contained in
    \begin{align*}
        \Big\{ (\xi, \tilde\xi) \in T^*SM \times T^*SM : & \ \exists \, s \neq 0 \ \text{such that} \ (\pi_{T^*SM}(\xi), s) \in \supp(\psi) \
        \\ & \ \text{and} \ dp|_{(\pi_{T^*SM}(\xi), s)}^t \, \xi = d\Phi|_{(\pi_{T^*SM}(\xi), s)}^t \, \tilde\xi \Big\}.
    \end{align*}
\end{lemma}

\begin{proof}
    Since $R_{a, \psi}$ equals $p_* \circ (\psi\tilde a)^m \circ \Phi^*$, we know
    \begin{equation*}
        \WF^\prime(R_{a, \psi}) \s C_{p_*} \circ \WF^\prime((\psi\tilde a)^m) \circ C_{\Phi^*}.
    \end{equation*}
    Because the Schwartz kernel of $(\psi\tilde a)^m$ is smooth away from $\supp(\psi\tilde a \otimes \psi\tilde a)$, and $a \in C^\infty(\R)$ is supported away from $0$, we also know
    \begin{align*}
        \WF^\prime((\psi\tilde a)^m) \s \Big\{ \big( \hat\xi, \sigma, \hat\xi, \sigma \big) \in \Delta \big( T^*(SM \times \R) \big) : & \ \pi_{T^*\R}(\sigma) \neq 0 \ \text{and} \\ 
        & \ \left( \pi_{T^*SM}(\hat\xi), \pi_{T^*\R}(\sigma) \right) \in \supp(\psi) \Big\}.
    \end{align*}
    Putting these two containments together yields the result.
\end{proof}

This simple observation will be useful when we split $\pi_* \circ R_a \circ \pi^*$ into several pieces.


\subsection{Composition with $\pi^*$}

Recall that our goal is to understand the microlocal structure of $\pi_* \circ R_a \circ \pi^*$. Theorem \ref{R_aThm} demonstrated that $R_{a, \psi}$ is a Fourier integral operator for any $\psi \in C^\infty(SM \times \R)$. The next theorem shows the same is true of the operator
\begin{equation*}
    L_\psi := R_{a, \psi} \circ \pi^*.
\end{equation*}

\begin{thm} \label{compWithPullback}
    Let
    \begin{align*} 
        C_{L_\psi} = \big\{ (\xi, \tilde\eta) \in T^*SM & \times T^*M : \exists \, s \in \R \ \text{such that} \\
        & \ dp|_{(\pi_{T^*SM}(\xi), s)}^t \, \xi = d\Phi|_{(\pi_{T^*SM}(\xi), s)}^t \circ d\pi|_{\Phi(\pi_{T^*SM}(\xi), s)}^t \, \tilde\eta \big\}.
    \end{align*}        
    Then $L_\psi \in \I^{-(n + 1) / 4}(SM \times M, C_{L_\psi}^\prime)$.
\end{thm}

\begin{proof}
    Using Theorem \ref{R_aThm} and Lemma \ref{pushPullLemma}, one can check that $C_{L_\psi}$ equals $C_{R_{a, \psi}} \circ C_{\pi^*}$, so it suffices to show points (i)-(ii) in the proof of Theorem \ref{R_aThm} hold with \eqref{C1} replaced by
    \begin{equation} \label{C2}
        C := (C_{R_{a, \psi}} \times C_{\pi^*}) \cap (T^*SM \times \Delta(T^*SM) \times T^*M).
    \end{equation}
    (Since we omit point (iii), in general $C_{L_\psi}$ will only be a local canonical relation.)

    To begin proving point (i), let $V$ be the smooth rank-$n$ subbundle of $T^*SM$ whose fiber over each vector $v \in SM$ is
    \begin{equation*}
        V_v := \Range \left( d\pi|_v^t \right).
    \end{equation*}
    Similar to \eqref{dPhi^-t}, we can define a smooth map
    \begin{equation*}
        d\pi^{-t} : V \to T^*M
    \end{equation*}
    whose restriction to each fiber $V_v$ is the inverse $(d\pi|_v^t)^{-1} : V_v \to T_{\pi(v)}^*M$. Let
    \begin{equation*}
        \widetilde Z = \left( d\Phi^{-t} \right)^{-1}(V).
    \end{equation*}    
    Since $d\Phi^{-t} : Z \to T^*SM$ is a smooth submersion and $V$ is an embedded codimension-$(n - 1)$ submanifold of $T^*SM$, we know $\widetilde Z$ is an embedded codimension-$(n - 1)$ submanifold of $Z$.

    Using a similar argument as the proof of Theorem \ref{R_aThm}, we will show that
    \begin{equation*}
        \widetilde \OO := \OO \cap \widetilde Z
    \end{equation*}
    is an embedded submanifold of dimension $3n - 1$. Let $\tilde q : \widetilde Z \to \R$ be the restriction to $\widetilde Z$ of the function $q$ defined in \eqref{q}. Then $\widetilde\OO$ equals the level set $\tilde q^{-1}(0)$, so it suffices to show that $d\tilde q|_{\tilde\zeta}$ is nonzero for all $\tilde\zeta \in \widetilde Z$. Let $\pi_V(\tilde\zeta) = (v, s)$. Then for some $\tilde\eta \in T_{\gamma_v(s)}^*M$, we have 
    \begin{equation*}
        \tilde\zeta = d\Phi|_{(v, s)}^t \circ d\pi|_{\Phi(v, s)}^t \, \tilde\eta.
    \end{equation*}
    Choose slice coordinates for $SM$ near $\Phi(v, s)$, which yield corresponding coordinates for $M$ near $\gamma_v(s)$. Fix natural coordinates on $T^*M$. Then locally we can write $\tilde\eta = (\tilde\eta_1, \dots, \tilde\eta_n)$. Define a curve $\beta$ in $T^*SM$ as the composition of $d\Phi|_{(v, s)}^t \circ d\pi|_{\Phi(v, s)}^t$ and the curve
    \begin{equation*}
        \R \ni \tau \mapsto (\tilde\eta_1 + \tau, \tilde\eta_2, \dots, \tilde\eta_n) \in T_{\gamma_v(s)}^*M.
    \end{equation*}
    Then $\beta$ is a smooth curve in $\widetilde Z$ such that $\beta(0) = \tilde\zeta$. Moreover, similar to \eqref{dqNonzero}, 
    \begin{equation*}
        \left. \frac{d}{d\tau} \right|_{\tau = 0} (\tilde q \circ \beta)(\tau) \neq 0.
    \end{equation*}
    This proves that $\widetilde\OO$ is an embedded submanifold of dimension $3n - 1$.
    
    Now we can use $\widetilde\OO$ to parametrize $C$ via the map
    \begin{equation*}
        P_C : \widetilde\OO \to T^*SM \times T^*SM \times T^*SM \times T^*M
    \end{equation*}
    defined by
    \begin{equation*}
        P_C(\tilde\zeta) = \left( G_\xi^{-1}(\tilde\zeta), d\Phi^{-t}\tilde\zeta, d\Phi^{-t}\tilde\zeta, d\pi^{-t} \circ d\Phi^{-t}\tilde\zeta \right).
    \end{equation*}
    By the last paragraph of the proof of Theorem \ref{R_aThm}, the map
    \begin{equation*}
        \OO \ni \zeta \mapsto \left( G_\xi^{-1}(\zeta), d\Phi^{-t}\zeta \right) \in T^*SM \times T^*SM
    \end{equation*}
    is a smooth embedding, and hence so is its restriction to $\widetilde\OO$. Thus $P_C$ is a smooth embedding, so $C$ is an embedded submanifold of dimension $3n - 1$. 

    To complete the proof of (i), note that $C_{R_{a, \psi}} \times C_{\pi^*}$ is parametrized by the map
    \begin{equation*}
        P_{C_{R_{a, \psi}} \times C_{\pi^*}} : \OO \times V \to T^*SM \times T^*SM \times V \times T^*M
    \end{equation*}
    given by
    \begin{equation*}
        P_{C_{R_{a, \psi}} \times C_{\pi^*}}(\zeta, \theta) = \left( G_\xi^{-1}(\zeta), d\Phi^{-t}\zeta, \theta, d\pi^{-t}\theta \right).
    \end{equation*}
    Fix $c \in C$ and suppose $X \in T_c(C_{R_{a, \psi}} \times C_{\pi^*})$. Then there is a smooth curve
    \begin{equation*}
        \R \ni \tau \mapsto (\zeta(\tau), \theta(\tau)) \in \OO \times V
    \end{equation*}
    such that $P_{C_{R_{a, \psi}} \times C_{\pi^*}}(\zeta(0), \theta(0)) = c$ and the velocity of this curve at zero is $X$. As in the proof of Theorem \ref{R_aThm}, if $X \in T_c(T^*SM \times \Delta(T^*SM) \times T^*M)$ as well, then the derivatives of $d\Phi^{-t}\zeta$ and $\theta$ agree at $\tau = 0$ in any local coordinates. It follows that
    \begin{align*}
        X &= \left. \frac{d}{d\tau} \right|_{\tau = 0} P_{C_{R_{a, \psi}} \times C_{\pi^*}}(\zeta(\tau), \theta(\tau)) \\
        &= \left. \frac{d}{d\tau} \right|_{\tau = 0} P_C(\zeta(\tau)),
    \end{align*}
    which means $X \in T_c C$. Therefore the intersection \eqref{C2} is clean and the excess is zero.

    The proof of point (ii) is the same as Theorem \ref{R_aThm}, so we omit the details. Thus $C_{L_\psi}$ is a local canonical relation. Since $R_{a, \psi}$ and $\pi^*$ are Fourier integral operators of order $-1 / 2$ and $(1 - n) / 4$, respectively, we conclude that $L_\psi \in \I^{-(n + 1) / 4}(SM \times M, C_{L_\psi}^\prime)$.
\end{proof}

When $\psi = 1$, we will write $L$ and $C_L$ instead of $L_1$ and $C_{L_1}$. Then by Theorem \ref{initialDecomp}, 
\begin{equation*}
    \A = \A_{2\alpha} + \A_0 + \pi_* \circ L.
\end{equation*}
Hence the microlocal analysis of $\A$ reduces to understanding the composition $\pi_* \circ L$.

The main difficulty in this case is that $C_{\pi_*} \circ C_L$ may have multiple connected components. One component corresponds to a smoothing operator, and the others appear only when there are conjugate points. If we rule out certain types of conjugate points, then these additional components give rise to Fourier integral operators whose canonical relations and orders can be determined. This is the content of Theorem \ref{mainThm}, our main result. In the next subsection, we will provide the additional definitions and lemmas needed to state and prove it. 


\subsection{Conjugate Pairs}

Though conjugate points along a geodesic are often defined in terms of vanishing Jacobi fields, it will be more convenient to work with the corresponding velocity vectors of the geodesic instead. This leads us to the notion of a conjugate pair.

\begin{defn} \label{conjugateDef}
    We call $(v, s) \in SM \times (\R \sm \{ 0 \})$ a \textbf{conjugate pair} if
    \begin{equation*}
        K_{(v, s)} := \ker \left( d\pi|_{\Phi(v, s)} \circ d_v\Phi|_{(v, s)} \right) \cap \ker(d\pi|_v) \neq \{ 0 \}.
    \end{equation*}
    If the dimension of $K_{(v, s)}$ is $1 \leq k \leq n - 1$, then $(v, s)$ is a \textbf{conjugate pair of order} $\boldsymbol k$. The \textbf{set of regular conjugate pairs of order} $\boldsymbol k$, denoted by $\C_{R, k}$, is the set of conjugate pairs which have a neighborhood $U$ in $SM \times \R$ such that all other conjugate pairs in $U$ have order $k$. The \textbf{set of singular conjugate pairs}, denoted by $\C_S$, is the set of conjugate pairs which are not in $\C_{R, k}$ for any $k$.
\end{defn}

By Lemma $3$ in \cite{holman2018microlocal}, Definition \ref{conjugateDef} is equivalent to the traditional definition of conjugate points in terms of vanishing Jacobi fields along a geodesic.

The crucial assumption in Theorem \ref{mainThm} is that there are no singular conjugate pairs. This matters because, as the next lemma shows, the set of regular conjugate pairs of order $k$ is a smooth manifold, which may not be true of the set of all conjugate pairs.

\begin{prop} \label{CRkManifold}
    For each integer $1 \leq k \leq n - 1$, the set $\C_{R, k}$ is an embedded $(2n - 1)$-dimensional submanifold of $SM \times \R$, and the set
    \begin{equation*}
        E_{R, k} := \left\{ \big( (v, s), X \big) \in \C_{R, k} \times TSM : X \in K_{(v, s)} \right\}
    \end{equation*}
    is a smooth vector bundle of rank $k$ over $\C_{R, k}$.
\end{prop}

\begin{proof}
    For the first point, it is enough to show that each point in $\C_{R, k}$ has a neighborhood $U$ in $SM \times \R$ such that $\C_{R, k} \cap U$ is an embedded submanifold of dimension $2n - 1$. To prove this local statement, we will extend the methods of \cite{warner1965conjugate}.
    
    Fix $(v, s) \in \C_{R, k}$. Let $d_F\exp$ be the differential in the fiber variables of the exponential map $\exp : TM \to M$. By \cite{warner1965conjugate}, we can find coordinate neighborhoods $W_1$ of $sv$ in $TM$ and $W_2$ of $\exp(sv)$ in $M$ such that the $(k - 1)$st elementary symmetric polynomial in the eigenvalues of $d_F\exp$ (denoted by $\sigma_{k - 1}$) has nonzero derivative in the radial direction. Then $\sigma_{k - 1}^{-1}(0)$ is an embedded $(2n - 1)$-dimensional submanifold of $TM \sm \{ 0 \}$, and it equals the set of vectors in $W_1$ with conjugate points of order $k$ or higher in $W_2$. 
    
    Choose a neighborhood $U$ of $(v, s)$ in $SM \times \R$ such that all other conjugate pairs in $U$ have order $k$. Supposing without loss of generality that $s > 0$, we may assume $U \s SM \times (0, \infty)$. Consider the smooth map $f : TM \sm \{ 0 \} \to SM \times \R$ defined by
    \begin{equation*}
        f(w) = \left( \frac{w}{|w|_g}, |w|_g \right).
    \end{equation*}
    Then $f$ is a smooth immersion and satisfies
    \begin{equation*}
        f \left( \sigma_{k - 1}^{-1}(0) \cap f^{-1}(U) \right) = \C_{R, k} \cap U.
    \end{equation*}
    Since $f|_{\sigma_{k - 1}^{-1}(0) \cap f^{-1}(U)}$ has a continuous inverse defined on its image by
    \begin{equation*}
        \C_{R, k} \cap U \ni (\tilde v, \tilde s) \mapsto \tilde s\tilde v \in TM \sm \{ 0 \},
    \end{equation*}
    it follows that $\C_{R, k} \cap U$ is an embedded submanifold of dimension $2n - 1$.

    To prove the second point, let $\V_{R, k}$ be the pullback of $\ker(d\pi)$ by the map
    \begin{equation*}
        \C_{R, k} \ni (v, s) \mapsto v \in SM.
    \end{equation*}
    Then $E_{R, k}$ is the kernel of the smooth bundle homomorphism
    \begin{equation*}
        \V_{R, k} \ni \big( (v, s), X \big) \mapsto \left( \Phi(v, s), \, d\pi|_{\Phi(v, s)} \circ d_v\Phi|_{(v, s)} X \right) \in TSM.
    \end{equation*}
    This map has constant rank $n - 1 - k$, and we can view it as a bundle homomorphism over $\C_{R, k}$ by pulling back $TSM$ by $\Phi$. Hence $E_{R, k}$ is a smooth rank-$k$ subbundle of $\V_{R, k}$.
\end{proof}

Next we will turn $TM$ into a symplectic manifold and make some remarks. Let $\omega$ be the canonical symplectic form on $T^*M$, and let $\flat_g : TM \to T^*M$ be the musical isomorphism induced by the metric $g$. Then we can define a symplectic form $\omega_g$ on $TM$ by
\begin{equation*}
    \omega_g(X, Y) = \omega(d\flat_g X, d\flat_g Y), \quad X, Y \in T(TM).
\end{equation*}
Then $\tilde\Phi(\cdot, s)$ is a symplectomorphism for each $s \in \R$, the kernel of $d\pi_{TM}|_v$ is a Lagrangian subspace of $T_v(TM)$ for each $v \in TM$, and in natural coordinates $\left( x^i, v^i \right)$ on $TM$ we have
\begin{equation} \label{omega_g_local}
    \omega_g = \xi^\ell \frac{\partial g_{i\ell}}{\partial x^j} \, dx^j \wedge dx^i + g_{ij} \, dv^j \wedge dx^i.
\end{equation}
In turn, $\omega_g$ induces a smooth bundle isomorphism $\flat_{\omega_g} : T(TM) \to T^*(TM)$ defined by
\begin{equation*}
    [\flat_{\omega_g}(X)](Y) = \omega_g(X, Y) := (X \into \omega_g)(Y),
\end{equation*}
where $X \into \omega_g$ is interior multiplication by $X$. We will denote the inverse of $\flat_{\omega_g}$ by $\sharp_{\omega_g}$.

The next lemma defines a smooth bundle homomorphism that will help us describe the canonical relations of the various pieces of $\pi_* \circ L$.

\begin{lemma} 
    For each integer $1 \leq k \leq n - 1$, there is a smooth bundle homomorphism
    \begin{equation*}
        F_k : E_{R, k} \to T^*(M \times M) = T^*M \times T^*M
    \end{equation*}
    defined by the requirement that
    \begin{equation} \label{F_k}
        \left( d\iota_{SM}|_v X \into \omega_g, \left( d\iota_{SM}|_{\Phi(v, s)} \circ d_v\Phi|_{(v, s)} X \right) \into \omega_g \right) = d\pi_{T(M \times M)}|_{(v, \Phi(v, s))}^t \, F_k \big( (v, s), X \big).
    \end{equation}
\end{lemma}

The proof of this result is essentially the same as that of Lemma $4$ in \cite{holman2018microlocal}, so we do not repeat the details here.    

Our final lemma is the key geometric tool in the proof of Theorem \ref{mainThm}. It will allow us to split $C_{\pi_*} \circ C_L$ into different pieces corresponding to different orders of conjugate pairs, each of which is associated with a Fourier integral operator.

\begin{lemma} \label{geometricLemma}
    Let $(v, s) \in SM \times \R$, $\tilde v = \Phi(v, s)$, $\eta \in T_{\pi(v)}^*M$, and $\tilde\eta \in T_{\pi(\tilde v)}^*M$. Then
    \begin{equation} \label{conjugateCondition}
        dp|_{(v, s)}^t \circ d\pi|_v^t \, \eta = d\Phi|_{(v, s)}^t \circ d\pi|_{\tilde v}^t \, \tilde\eta
    \end{equation}
    if and only if
    \begin{equation} \label{conjugateConclusion}
         d\pi|_v^t \, \eta = d_v\Phi|_{(v, s)}^t \circ d\pi|_{\tilde v}^t \, \tilde\eta \quad \text{and} \quad \eta(v) = \tilde\eta(\tilde v) = 0.
    \end{equation}
    If \eqref{conjugateConclusion} holds and $s \neq 0$ then $(v, s)$ is a conjugate pair, and if $(v, s) \in \C_{R, k}$ then $(\eta, \tilde\eta) \in F_k(E_{R, k})$. Conversely, if $(\eta, \tilde\eta) \in F_k(E_{R, k})$ then \eqref{conjugateConclusion} holds for some $(v, s) \in \C_{R, k}$.
\end{lemma}

\begin{proof}
    To see that \eqref{conjugateCondition} and \eqref{conjugateConclusion} are equivalent, first note that
    \begin{align*}
        dp|_{(v, s)}^t \circ d\pi|_v^t \, \eta &= \left( d\pi|_v^t \, \eta, 0 \right), \\
        d\Phi|_{(v, s)}^t \circ d\pi|_{\tilde v}^t \, \tilde\eta &= \left( d_v\Phi|_{(v, s)}^t \circ d\pi|_{\tilde v}^t \, \tilde\eta, \tilde\eta(\tilde v) \right),
    \end{align*}
    by \eqref{G} and the fact that $d\pi|_{\tilde v}^t \, \tilde\eta (\dot\Phi(v, s)) = \tilde\eta(\tilde v)$. Hence \eqref{conjugateCondition} holds if and only if
    \begin{equation*}
        d\pi|_v^t \, \eta = d_v\Phi|_{(v, s)}^t \circ d\pi|_{\tilde v}^t \, \tilde\eta \quad \text{and} \quad \tilde\eta(\tilde v) = 0.
    \end{equation*}
    Thus \eqref{conjugateConclusion} implies \eqref{conjugateCondition}. For the converse, just apply both sides of \eqref{conjugateCondition} to the vector $(\dot\Phi(v, 0), 0) \in T_{(v, s)}(SM \times \R)$ and deduce that $\eta(v) = \tilde\eta(\tilde v)$.

    Before addressing the claims about conjugate pairs, let us make a few observations. It will be useful to work with $TM$ rather than $SM$. To connect the two, observe that
    \begin{align} 
        d\pi|_v^t &= d\iota_{SM}|_v^t \circ d\pi_{TM}|_v^t, \label{transposed1} \\
        d_v\Phi|_{(v, s)}^t \circ d\iota_{SM}|_{\tilde v}^t &= d\iota_{SM}|_v^t \circ d_v\tilde\Phi|_{(v, s)}^t, \label{transposed2}
    \end{align}
    due to the identities $\pi = \pi_{TM} \circ \iota_{SM}$ and $\iota_{SM} \circ \Phi = \tilde\Phi(\iota_{SM}(\cdot), \cdot)$. Let
    \begin{equation} \label{Xdef}
        X = \left( d\pi_{TM}|_v^t \, \eta \right)^{\sharp_{\omega_g}} \in T_v(TM).
    \end{equation}
    Equivalently, applying $\flat_{\omega_g}$, we have
    \begin{equation} \label{computation1}
        X \into \omega_g = d\pi_{TM}|_v^t \, \eta.
    \end{equation}
    In the next paragraph, we will prove the following analogue of the first condition in \eqref{conjugateConclusion}:
    \begin{equation} \label{conjugateConclusionTM}
        d\pi_{TM}|_v^t \, \eta = d_v\tilde\Phi|_{(v, s)}^t \circ d\pi_{TM}|_{\tilde v}^t \, \tilde\eta.
    \end{equation}
    Assuming \eqref{conjugateConclusionTM} for the moment, the fact that $\tilde\Phi(\cdot, s)$ is a symplectomorphism implies 
    \begin{equation} \label{computation2}
        d\pi_{TM}|_{\tilde v}^t \, \tilde\eta = d_v\tilde\Phi|_{(v, s)} X \into \omega_g.
    \end{equation}

    Now suppose \eqref{conjugateConclusion} holds and $s \neq 0$. We will divide the proof that $(v, s)$ is a conjugate pair into three steps. The first one is to prove \eqref{conjugateConclusionTM}. By \eqref{transposed1}, \eqref{transposed2}, and \eqref{conjugateConclusion},
    \begin{equation*}
        d\iota_{SM}|_v^t \circ d\pi_{TM}|_v^t \, \eta = d\iota_{SM}|_v^t \circ d_v\tilde\Phi|_{(v, s)}^t \circ d\pi_{TM}|_{\tilde v}^t \, \tilde\eta.
    \end{equation*}
    Since $\ker(d\iota_{SM}|_v^t)$ is the span of the differential of $TM \ni w \mapsto |w|_g^2$, this means
    \begin{equation*}
        d\pi_{TM}|_v^t \, \eta = d_v\tilde\Phi|_{(v, s)}^t \circ d\pi_{TM}|_{\tilde v}^t \, \tilde\eta + \frac{\tau}{2} \left. d(|w|_g^2) \right|_v
    \end{equation*}
    for some $\tau \in \R$. Applying both sides to a radial vector $r \in T_v(TM)$, we find
    \begin{equation*}
        0 = \tilde\eta \left( d\pi_{TM}|_{\tilde v} \circ d_v\tilde\Phi|_{(v, s)} \, r \right) + \tau.
    \end{equation*}
    Since $\tilde\eta(\tilde v) = 0$ by assumption and the vector $d\pi_{TM}|_{\tilde v} \circ d_v\tilde\Phi|_{(v, s)} \, r$ is parallel to $\tilde v$, this implies that $\tau = 0$ and hence completes the proof of \eqref{conjugateConclusionTM}.

    The second step is to show that \eqref{Xdef} is in $\Range(d\iota_{SM}|_v)$. Since \eqref{computation1} implies that $X \into \omega_g$ vanishes on $\ker(d\pi_{TM}|_v)$, which is a Lagrangian subspace, $X$ must be in $\ker(d\pi_{TM}|_v)$. Choose normal coordinates $\left( x^i \right)$ on $M$ centered at $\pi(v)$ such that
    \begin{equation} \label{vNormalCoordinates}
        v = d\pi_{TM}|_v \, \frac{\partial}{\partial x^1},
    \end{equation}
    and let $\left( x^i, v^i \right)$ be natural coordinates on $TM$. Then $\Range(d\iota_{SM}|_v)$ is the span of the vectors $\partial / \partial v^2, \dots, \partial / \partial v^n$. Since $X$ is in $\ker(d\pi_{TM}|_v)$, we can write
    \begin{equation*}
        X = a^j \frac{\partial}{\partial v^j}
    \end{equation*}
    for some $a^j \in \R$. Using \eqref{computation1}, \eqref{vNormalCoordinates}, and the assumption $\eta(v) = 0$, we find
    \begin{equation*}
        \omega_g \left( X, \frac{\partial}{\partial x^1} \right) = 0.
    \end{equation*}
    Using \eqref{omega_g_local} and the fact that $g_{ij} = \delta_{ij}$ at $\pi(v)$, this implies $a^1 = 0$. Hence $X$ is in $\Range(d\iota_{SM}|_v)$, so we can define the vector $d\iota_{SM}|_v^{-1} X \in T_v SM$.

    The third step is to show that $d\iota_{SM}|_v^{-1} X$ is in $K_{(v, s)}$, meaning
    \begin{equation*} 
        d\iota_{SM}|_v^{-1} X \in \ker \left( d\pi|_{\tilde v} \circ d_v\Phi|_{(v, s)} \right) \cap \ker(d\pi|_v). 
    \end{equation*}
    Since $X$ is in $\ker(d\pi_{TM}|_v)$ and $d\pi|_v = d\pi_{TM}|_v \circ d\iota_{SM}|_v$, we know $d\iota_{SM}|_v^{-1} X$ is in $\ker(d\pi|_v)$. Hence it is enough to prove that
    \begin{equation*}
        d\pi|_{\tilde v} \circ d_v\Phi|_{(v, s)} \circ d\iota_{SM}|_v^{-1} X = 0.
    \end{equation*}
    By the transposes of \eqref{transposed1} and \eqref{transposed2}, this is equivalent to showing
    \begin{equation*} 
        d\pi_{TM}|_{\tilde v} \circ d_v\tilde\Phi|_{(v, s)} X = 0.
    \end{equation*}
    But \eqref{computation2} implies that $d_v\tilde\Phi|_{(v, s)} X \into \omega_g$ vanishes on the Lagrangian subspace $\ker(d\pi_{TM}|_{\tilde v})$, so $d_v\tilde\Phi|_{(v, s)} X$ is indeed in $\ker(d\pi_{TM}|_{\tilde v})$. This proves that $(v, s)$ is a conjugate pair.

    Now suppose $(v, s) \in \C_{R, k}$. Since $d\iota_{SM}|_v^{-1} X$ is in $K_{(v, s)}$, this means
    \begin{equation*}
        \left( (v, s), d\iota_{SM}|_v^{-1} X \right) \in E_{R, k}.
    \end{equation*}
    Using \eqref{computation1} in the first line below and \eqref{transposed2} (transposed) and \eqref{computation2} in the second, we find
    \begin{align*}
        \left( d\iota_{SM}|_v \circ d\iota_{SM}|_v^{-1} X \right) \into \omega_g &= X \into \omega_g = d\pi_{TM}|_v^t \, \eta, \\
        \left( d\iota_{SM}|_{\tilde v} \circ d_v\Phi|_{(v, s)} \circ d\iota_{SM}|_v^{-1} X \right) \into \omega_g &= d_v\tilde\Phi|_{(v, s)} X \into \omega_g = d\pi_{TM}|_{\tilde v}^t \, \tilde\eta.
    \end{align*}
    Hence $(\eta, \tilde\eta) = F_k \big( (v, s), d\iota_{SM}|_v^{-1} X \big)$, which proves that $(\eta, \tilde\eta) \in F_k(E_{R, k})$.

    Conversely, suppose $(\eta, \tilde\eta) = F_k \big( (v, s), X \big)$. Unpacking \eqref{F_k}, this means
    \begin{align}
        d\iota_{SM}|_v X \into \omega_g &= d\pi_{TM}|_v^t \, \eta, \label{F_kEquivalent1} \\
        \left( d\iota_{SM}|_{\tilde v} \circ d_v\Phi|_{(v, s)} X \right) \into \omega_g &= d\pi_{TM}|_{\tilde v}^t \, \tilde\eta \label{F_kEquivalent2}.
    \end{align} 
    Let $Y \in T_{\tilde v}(TM)$. Then \eqref{F_kEquivalent2} and the transpose of \eqref{transposed2} imply that
    \begin{equation*}
        \omega_g \left( d_v\tilde\Phi|_{(v, s)} \circ d\iota_{SM}|_v X, Y \right) = d\pi_{TM}|_{\tilde v}^t \, \tilde\eta(Y).
    \end{equation*}
    Using that $\tilde\Phi(\cdot, s)$ is a symplectomorphism together with \eqref{F_kEquivalent1}, we find
    \begin{equation*}
        \left( d_v\tilde\Phi|_{(v, s)}^t \right)^{-1} \circ d\pi_{TM}|_v^t \, \eta(Y) = d\pi_{TM}|_{\tilde v}^t \, \tilde\eta(Y).
    \end{equation*}
    Hence \eqref{conjugateConclusionTM} holds in this direction of the proof as well. Applying $d\iota_{SM}|_v^t$ to both sides of that equation and rewriting with \eqref{transposed1} and \eqref{transposed2} establishes the first condition in \eqref{conjugateConclusion}.
    
    Now take the same natural coordinates on $TM$ described above \eqref{vNormalCoordinates}. Then
    \begin{equation*}
        d\iota_{SM}|_v X = \sum_{j = 2}^n a^j \frac{\partial}{\partial v^j}
    \end{equation*}
    for some $a^j \in \R$. By \eqref{vNormalCoordinates}, \eqref{F_kEquivalent1}, and \eqref{omega_g_local}, we get
    \begin{equation*}
        \eta(v) = d\pi_{TM}|_v^t \, \eta \left( \frac{\partial}{\partial x^1} \right) = (d\iota_{SM}|_v X \into \omega_g) \left( \frac{\partial}{\partial x^1} \right) = \omega_g \left( d\iota_{SM}|_v X, \frac{\partial}{\partial x^1} \right) = 0,
    \end{equation*}
    and $\tilde\eta(\tilde v) = 0$ by a similar argument. This establishes the second condition in \eqref{conjugateConclusion}.
\end{proof}


\subsection{Final Decomposition of the L\'evy Generator}

Now we can state our main result. It refines Theorem \ref{initialDecomp} by decomposing $\pi_* \circ R_a \circ \pi^*$ into a smoothing operator and a sum of Fourier integral operators, assuming there are no singular conjugate pairs. 

\begin{thm} \label{mainThm}
    Suppose $\C_S = \emptyset$. Then for $k = 1$ to $n - 1$, the sets
    \begin{equation*}
        C_{A_k} = F_k(E_{R, k}) \s T^*(M \times M)
    \end{equation*}
    are either local canonical relations or empty. Let $C_{A_{k, 1}}, \dots, C_{A_{k, M_k}}$ be the connected components of $C_{A_k}$. Let $\A_{2\alpha}$ and $\A_0$ be the pseudodifferential operators from Theorem \ref{initialDecomp}. Then
    \begin{equation*}
        \A = \A_{2\alpha} + \A_0 + \A_{-\infty} + \sum_{k = 1}^{n - 1} \left( \sum_{m = 1}^{M_k} A_{k, m} \right),
    \end{equation*}
    where $\A_{-\infty}$ is a smoothing operator, and for each $k$ either 
    \begin{equation*}
        A_{k, m} \in \I^{-(n - k + 1) / 2}(M \times M, C_{A_{k, m}}^\prime),
    \end{equation*}
    or $M_k = 1$ and $A_{k, 1} = 0$ if $C_{A_k} = \emptyset$.
\end{thm}

\begin{proof}
    We must decompose $\pi_* \circ L$ into a smoothing operator and a sum of Fourier integral operators. Though the clean intersection calculus does not directly apply to $\pi_* \circ L$, we will cut up this operator so that it applies to each separate piece. 
    
    First let us describe $C_{\pi_*} \circ C_L$. By Lemma \ref{pushPullLemma} and Theorem \ref{compWithPullback},
    \begin{align*} 
        C_{\pi_*} \circ C_L = \Big\{ (\eta, \tilde\eta) \in T^*M & \times T^*M : \exists \, (v, s) \in SM \times \R \ \text{such that} \\
        & \ dp|_{(v, s)}^t \circ d\pi|_v^t \, \eta = d\Phi|_{(v, s)}^t \circ d\pi|_{\Phi(v, s)}^t \, \tilde\eta \Big\}.
    \end{align*}        
    The requirement in $C_{\pi_*} \circ C_L$ is precisely \eqref{conjugateCondition}, which is equivalent to \eqref{conjugateConclusion} by Lemma \ref{geometricLemma}. We can also use this lemma to cut up $C_{\pi_*} \circ C_L$ into several pieces according to different orders of conjugate pairs. Indeed, if we take $s = 0$ and any $v \in SM$ such that $\eta(v) = 0$, then $v$, $s$, and $\eta$ satisfy \eqref{conjugateConclusion}, so one piece is the diagonal
    \begin{equation*}
        \Delta := \{ (\eta, \eta) \in T^*M \times T^*M \}.
    \end{equation*}
    If $(\eta, \tilde\eta) \in C_{\pi_*} \circ C_L$ and $\eta \neq \tilde\eta$, then \eqref{conjugateConclusion} is satisfied for some $(v, s) \in SM \times \R$, and $s \neq 0$ because $d\pi|_v^t$ is injective. Then Lemma \ref{geometricLemma} and the assumption $\C_S = \emptyset$ imply that
    \begin{equation*}
        C_{\pi_*} \circ C_L = \Delta \cup \left( \bigcup_{k = 1}^{n - 1} C_{A_k} \right).
    \end{equation*}
    
    Our goal is to write $\pi_* \circ L$ as a sum of Fourier integral operators, each having a canonical relation contained in a single set of this union. Since $\C_S = \emptyset$, we can find open subsets $U_k$ of $SM \times \R$ with disjoint closures such that $\C_{R, k} \s U_k$ for $k = 1$ to $n - 1$. Write
    \begin{equation*}
        U_k = \bigcup_{m = 1}^{M_k} U_{k, m},
    \end{equation*}
    where each open set $U_{k, m}$ contains exactly one of the connected components of $\C_{R, k}$. Then we can construct a partition of unity $\{ \psi_{k, m} \}$ on $SM \times \R$ such that 
    \begin{equation*}
        \supp(\psi_{0, 1}) \s (SM \times \R) \sm \left( \bigcup_{k = 1}^{n - 1} \C_{R, k} \right),
    \end{equation*}
     and $\supp(\psi_{k, m}) \s U_{k, m}$ for $k = 1$ to $n - 1$ and $m = 1$ to $M_k$. Setting $M_0 = 1$, we have
    \begin{equation*}
        \pi_* \circ L = \sum_{k = 0}^{n - 1} \sum_{m = 1}^{M_k} \pi_* \circ L_{\psi_{k, m}}.
    \end{equation*}
    By Lemma \ref{R_aWF}, the twisted wave front set of $L_{\psi_{k, m}}$ is contained in
    \begin{align*} 
        C_{L, k, m} := \big\{ (\xi, \tilde\eta) \in T^*SM \times T^*M : \exists \, s \neq 0 \ \text{such that} \ (\pi_{T^*SM}(\xi), s) \in \supp(\psi_{k, m}) \\
        \ \text{and} \ dp|_{(\pi_{T^*SM}(\xi), s)}^t \, \xi = d\Phi|_{(\pi_{T^*SM}(\xi), s)}^t \circ d\pi|_{\Phi(\pi_{T^*SM}(\xi), s)}^t \, \tilde\eta \big\}.
    \end{align*}
    Therefore we obtain smoothing operators except near points in
    \begin{align*} 
            (C_{\pi_*} \circ C_{L, k, m})^\prime = \Big\{ & (\eta, -\tilde\eta) \in T^*M \times T^*M : \exists \, (v, s) \in SM \times (\R \sm \{ 0 \}) \ \text{such that} \\
            & \ (v, s) \in \supp(\psi_{k, m}) \ \text{and} \ dp|_{(v, s)}^t \circ d\pi|_v^t \, \eta = d\Phi|_{(v, s)}^t \circ d\pi|_{\Phi(v, s)}^t \, \tilde\eta \Big\}.
    \end{align*}        
    Since Lemma \ref{geometricLemma} implies that $C_{\pi_*} \circ C_{L, 0, 1}$ is empty, 
    \begin{equation*}
        \A_{-\infty} := \pi_* \circ L_{\psi_{0, 1}}
    \end{equation*}
    is a smoothing operator. Lemma \ref{geometricLemma} also implies that $C_{\pi_*} \circ C_{L, k, m} = C_{A_{k, m}}$ for $k = 1$ to $n - 1$ and $m = 1$ to $M_k$. If any of these compositions are empty, we can absorb the corresponding operator into $\A_{-\infty}$ and set $A_{k, m} = 0$.
    
    It remains to show that if $C_{\pi_*} \circ C_{L, k, m}$ is nonempty, then the clean intersection calculus applies to $\pi_* \circ L_{\psi_{k, m}}$. We again refer to points (i)-(ii) in the proof of Theorem \ref{R_aThm}. (Since we omit point (iii), we only obtain local canonical relations in general.) Define
    \begin{equation*}
        C_{k, m} = (C_{\pi_*} \times C_{L, k, m}) \cap (T^*M \times \Delta(T^*SM) \times T^*M).
    \end{equation*}
    Let $E_{R, k, m}$ be the restriction of $E_{R, k}$ to the $m$th connected component of $\C_{R, k}$, and let $F_k^\ell$ be the $\ell$th component function of $F_k$ for $\ell = 1, 2$. Then Lemma \ref{geometricLemma} implies that $C_{k, m}$ is a connected component of
    \begin{equation} \label{componentOf}
        \begin{aligned}
            \Big\{ \left( F_k^1 \big( (v, s), X \big), d\iota_{SM}|_v^t(d\iota_{SM}|_v X \into \omega_g), d\iota_{SM}|_v^t(d\iota_{SM}|_v X \into \omega_g), F_k^2 \big( (v, s), X \big) \right) : \\
            \big( (v, s), X \big) \in E_{R, k, m} \Big\}.
        \end{aligned}
    \end{equation}
    Thus, to show $C_{k, m}$ is an embedded submanifold, it is enough to prove that
    \begin{equation} \label{ERkmEmbedding}
        E_{R, k, m} \ni \big( (v, s), X \big) \mapsto \left( v, d\iota_{SM}|_v^t(d\iota_{SM}|_v X \into \omega_g) \right) \in T^*SM
    \end{equation}
    is a smooth embedding. Let $p_{k, m}$ be the projection of the $m$th connected component of $\C_{R, k}$ onto $SM$. Let $\V_{R, k, m}$ be the pullback of $\ker(d\pi)$ by $p_{k, m}$. Then $E_{R, k, m}$ is a smooth subbundle of $\V_{R, k, m}$ (as in the proof of Proposition \ref{CRkManifold}), so it suffices to show that the extension of \eqref{ERkmEmbedding} to $\V_{R, k, m}$ is a smooth embedding. Observe that the bundle homomorphism
    \begin{equation*}
        E_{R, k, m} \ni \big( (v, s), X \big) \mapsto (v, X) \in TSM
    \end{equation*}
    covers $p_{k, m}$ and is a smooth embedding (because it is a proper injective immersion). This implies that $p_{k, m}$ is a smooth embedding, so $\V_{R, k, m}$ is smoothly isomorphic to the restriction of $\ker(d\pi)$ to $\Range(p_{k, m})$. Hence it suffices to show that
    \begin{equation*}
        \ker(d\pi)|_{\Range(p_{k, m})} \ni (v, X) \mapsto \left( v, d\iota_{SM}|_v^t(d\iota_{SM}|_v X \into \omega_g) \right) \in T^*SM
    \end{equation*}
    is a smooth embedding. But \eqref{omega_g_local} implies that this map is injective in each fiber, and hence a smooth bundle isomorphism onto its image. Therefore \eqref{ERkmEmbedding} is a smooth embedding, which proves that $C_{k, m}$ is a connected embedded submanifold of dimension $2n + k - 1$.

    Now let $c = (c_1, c_2, c_3, c_4) \in C_{k, m}$ be arbitrary, and consider the set
    \begin{align*}
        D_{k, m} &:= T_c(C_{\pi_*} \times C_{L, k, m}) \cap T_c(T^*M \times \Delta(T^*SM) \times T^*M) \\
        &\s T_{c_1}(T^*M) \times T_{c_2}(T^*SM) \times T_{c_3}(T^*SM) \times T_{c_4}(T^*M).
    \end{align*}
    Since $T_c C_{k, m}$ has dimension $2n + k - 1$ and is contained in $D_{k, m}$, the intersection is clean if the dimension of $D_{k, m}$ is at most $2n + k - 1$. Suppose $(Y_1, Y_2, Y_3, Y_4) \in D_{k, m}$. Then $Y_2 = Y_3$, and an examination of $C_{L, k, m}$ shows that $Y_3$ determines $Y_4$. Hence the dimension of $D_{k, m}$ is at most $3n - 1$ (the dimension of $C_{\pi_*}$). But for fixed $(v, s)$, the set
    \begin{equation*}
        \left\{ d\iota_{SM}|_v^t(d\iota_{SM}|_v X \into \omega_g) : \big( (v, s), X \big) \in E_{R, k, m} \right\}
    \end{equation*}
    is a $k$-dimensional vector space contained in
    \begin{equation*}
        \left\{ d\pi|_v^t \, \eta : \eta \in T_{\pi(v)}^*M \right\}.
    \end{equation*} 
    Therefore the dimension of $D_{k, m}$ is at most $3n - 1 - (n - k) = 2n + k - 1$, so the intersection is clean with excess $k - 1$. Because $C_{k, m}$ is a component of \eqref{componentOf}, the projection map
    \begin{equation*}
        \pi_{k, m} : C_{k, m} \to T^*M \times T^*M
    \end{equation*}
    is proper by the same argument as the proof of Theorem \ref{R_aThm}. Since $\pi_*$ and $L_{\psi_{k, m}}$ are Fourier integral operators of order $(1 - n) / 4$ and $-(n + 1) / 4$, respectively, we conclude that
    \begin{equation*}
        A_{k, m} := \pi_* \circ L_{\psi_{k, m}}
    \end{equation*}
    is in $\I^{-(n - k + 1) / 2}(M \times M, C_{A_{k, m}}^\prime)$ whenever the set $C_{A_{k, m}}$ is nonempty.
\end{proof}

Two special cases of Theorem \ref{mainThm} are worth mentioning. First, if $(M, g)$ is Anosov then it has no conjugate points \cite{ruggiero1991creation}, so each Fourier integral operator $A_{k, m}$ is zero and we recover Theorem $1.6$ in \cite{chaubet2022geodesic}. Second, Theorem \ref{mainThm} covers all possibilities in two dimensions because singular conjugate pairs cannot exist (since conjugate pairs can only have order $1$). In higher dimensions, the generic case includes singular conjugate pairs \cite{arnol1972normal, klok1983generic}.


\printbibliography[heading=bibintoc]

\end{document}